\documentclass[11pt]{amsart}

\usepackage{amssymb}
\usepackage{amsthm}
\usepackage{amsmath}
\usepackage[all]{xy}
\usepackage[margin=1in]{geometry}

\usepackage[usenames]{color}
\definecolor{ANDREW}{RGB}{255,127,0}

\usepackage{tikz}
\usetikzlibrary{arrows,calc}

\usepackage{tikz-cd}

\usepackage[foot]{amsaddr}

\usepackage{hyperref}

\theoremstyle{plain}
\newtheorem{proposition}{Proposition}[section]
\newtheorem{theorem}[proposition]{Theorem}
\newtheorem{lemma}[proposition]{Lemma}
\newtheorem{corollary}[proposition]{Corollary}
\theoremstyle{definition}

\newtheorem{observation}[proposition]{Observation}
\theoremstyle{remark}
\newtheorem{remark}[proposition]{Remark}

\newtheorem*{question}{Question}

\DeclareMathOperator{\Aut}{Aut}

\DeclareMathOperator{\Id}{Id}

\DeclareMathOperator{\End}{End}

\DeclareMathOperator{\id}{id}

\DeclareMathOperator{\dev}{dev}
\DeclareMathOperator{\dist}{dist}

\DeclareMathOperator{\Cc}{\mathcal{C}}

\DeclareMathOperator{\Mcc}{\mathcal{M}}

\DeclareMathOperator{\Oc}{\mathcal{O}}
\DeclareMathOperator{\Pc}{\mathcal{P}}

\DeclareMathOperator{\Bb}{\mathbb{B}}
\DeclareMathOperator{\Cb}{\mathbb{C}}
\DeclareMathOperator{\Db}{\mathbb{D}}

\DeclareMathOperator{\Nb}{\mathbb{N}}

\DeclareMathOperator{\Rb}{\mathbb{R}}

\DeclareMathOperator{\Zb}{\mathbb{Z}}

\newcommand{\abs}[1]{\left|#1\right|}

\newcommand{\norm}[1]{\left\|#1\right\|}
\newcommand{\wt}[1]{\widetilde{#1}}

\begin{document}

\title{A metric analogue of Hartogs' theorem}

\author[H. Gaussier]{Herv\'e Gaussier$^1$
}
\address{H. Gaussier: Univ. Grenoble Alpes, CNRS, IF, F-38000 Grenoble, France}
\email{herve.gaussier@univ-grenoble-alpes.fr}

\author[A. Zimmer]{Andrew Zimmer$^2$
}
\address{A. Zimmer: Department of Mathematics, University of Wisconsin-Madison, USA }
\email{amzimmer2@wisc.edu}
\date{\today}
\subjclass[2010]{}

\thanks{$^1\,$Partially supported by ERC ALKAGE}
\thanks{$^2\,$ Partially supported by grants DMS-2105580 and DMS-2104381 from the
National Science Foundation.}
\begin{abstract}
In this paper we prove a metric version of Hartogs' theorem where the holomorphic function is replaced by a locally symmetric Hermitian metric. As an application, we prove that if the Kobayashi metric on a strongly pseudoconvex domain with $\mathcal{C}^2$ smooth boundary is a K\"ahler metric, then the universal cover of the domain is the unit ball.
\end{abstract}

\maketitle

\section{Introduction}

Recall that Hartogs' theorem states that if $X$ is a Stein manifold with (complex) dimension at least two, $K \subset X$ is compact, and $X\setminus K$ is connected, then any holomorphic function $f : X \setminus K \rightarrow \Cb$ extends to a holomorphic function on all of $X$. In this paper we prove a metric version of Hartogs' theorem where the holomorphic function is replaced by a locally symmetric Hermitian metric.

To state our result precisely we need one technical definition. Given a non-compact manifold $X$, a compact set $K \subset X$ and a Riemannian metric $g_0$ on $X \setminus K$, it is always possible to find a metric $g$ on $X$ and a compact set $K' \subset X$ such that $K \subset K'$ and $g=g_0$ on $X \setminus K'$. Notice that if one such extension is a complete metric, then  all such extensions are complete and in this case we say that $g_0$  is \emph{complete at infinity}.

\begin{theorem}\label{thm:metric_hartog's} Suppose that $X$ is a Stein manifold with $\dim_{\Cb} X \geq 2$, $K \subset X$ is a compact subset where $X \setminus K$ is connected, and $g_0$ is a Hermitian metric on $X \setminus K$ which is complete at infinity. If $g_0$ is locally symmetric, then there exists a complete locally symmetric Hermitian metric  $g$ on $X$ such that $g=g_0$ on $X \setminus K$. 
\end{theorem} 

Using the fact that a K\"ahler metric with constant holomorphic sectional curvature is locally symmetric, we will prove the following corollary.

\begin{corollary}\label{cor:universal_cover_is_a_ball} Suppose that $X$ is a Stein manifold with $\dim_{\Cb} X \geq 2$ and $g$ is a complete K\"ahler metric on $X$. If there exists a compact set $K \subset X$ such that $g$ has constant negative holomorphic sectional curvature on $X \setminus K$, then the universal cover of $X$ is biholomorphic to the unit ball in $\Cb^{\dim_{\Cb} X}$. 
\end{corollary}

It is fairly easy to see that the assumption that $\dim_{\Cb} X \geq 2$ is necessary in Corollary~\ref{cor:universal_cover_is_a_ball} and hence also in Theorem~\ref{thm:metric_hartog's}. For instance, let $f : \Rb \rightarrow (0,\infty)$ be a smooth function such that $f(y) \geq 1$ when $\abs{y} \leq 1$ and $f(y) = \frac{1}{y^2}$ when $\abs{y} > 1$. Then the K\"ahler metric
$$
g_z(v,w) = f({\rm Im}(z)) {\rm Re}(v\bar{w})
$$
on $\Cb$ (here we identify $T_z \Cb \simeq \Cb$) is complete and has constant negative holomorphic sectional curvature on $ \{ x+iy \in \Cb : \abs{y} > 1\}$. Further, real translations act by isometries on $(\Cb, g)$. Hence the metric $g$ descends to a complete K\"ahler metric on $X = \Zb \backslash \Cb$ which has constant negative holomorphic sectional curvature outside a compact set.  Finally, from the Behnke-Stein theorem \cite{BS1949}, $X$ is a Stein manifold.

Recall that a simply connected non-positively curved K\"ahler manifold is Stein~\cite{GW1979}. In this special case, a proof of Corollary~\ref{cor:universal_cover_is_a_ball} was outlined by Greene~\cite[pp. 344]{G1982} and later established, using a different approach, by Seshadri-Verma~\cite{SV2006}. 

Seshadri-Verma~\cite{SV2006} also conjectured that if a simply connected non-positively curved K\"ahler manifold is locally symmetric outside a compact set, then it is biholomorphic to a Hermitian symmetric space. Theorem~\ref{thm:metric_hartog's} provides a positive answer to their conjecture.

\subsection{An application} As an application of Corollary~\ref{cor:universal_cover_is_a_ball} we study coincidences between the various classical invariant metrics on a bounded pseudoconvex domain in complex Euclidean space. In particular, such a domain  has a number of metrics: the Kobayashi metric, the Bergman metric, and the unique up to scaling K\"ahler-Einstein metric. All of these metrics coincide, up to a multiplicative constant, on the unit ball but there is no reason to think this would happen for a generic domain. This leads to the following natural question.

\begin{question}\label{prob:coincidence} What are the domains where the Kobayashi, Bergman, and K\"ahler-Einstein metrics are not pairwise distinct (up to scaling)? \end{question} 

This is a variant of an old and well known problem: In 1979, Cheng~\cite{C1979} conjectured that on a strongly pseudoconvex domain the Bergman metric is K\"ahler-Einstein if and only if the domain is biholomorphic to the unit ball. More generally, Yau~\cite[problem no. 44]{Yau1982} asked (in a slightly different form)  if it was possible to classify the pseudoconvex domains where the Bergman metric is K\"ahler-Einstein. 

Recently, Huang and Xiao established Cheng’s conjecture for strongly pseudoconvex domains with $\Cc^\infty$ boundary.

\begin{theorem}\cite{HX2021}\label{thm:Cheng} If $\Omega \subset \Cb^d$ is a bounded strongly pseudoconvex domain with $\Cc^\infty$ boundary, then the Bergman metric is K\"ahler-Einstein if and only if $\Omega$ is biholomorphic to the unit ball. 
\end{theorem}

\begin{remark} In the special case when $d=2$ and $\Omega$ is simply connected, Theorem~\ref{thm:Cheng} was established by Fu and Wong~\cite{FW1997}. \end{remark}

As an application of Corollary~\ref{cor:universal_cover_is_a_ball}  we investigate the bounded domains in $\Cb^d$ where the Kobayashi metric coincides up to scaling with the Bergman or K\"ahler-Einstein metric, or more generally is a K\"ahler metric. In this direction, the best result appears to be from the 1980's and is due to Stanton~\cite{S1983}: if either the Kobayashi metric or the Carath\'eodory metric is a smooth Hermitian metric and the two metrics coincide at a point, then the domain is biholomorphic to the unit ball. By deep work of Lempert~\cite{L1981} this in particular shows that the Kobayashi metric on a bounded convex domain is a K\"ahler metric if and only if the domain is biholomorphic to the unit ball. 

We also note that Burns-Shnider~\cite{BS1976} have constructed examples of bounded strongly pseudoconvex domains in $\Cb^d$ with non-trivial fundamental groups whose universal cover is the unit ball. For these examples, the Kobayashi metric and the K\"ahler-Einstein metric coincide up to scaling. So as one moves beyond the case of convex domains, more examples appear.

For strongly pseudoconvex domains, we prove the following general result. 

\begin{theorem}\label{thm:spcv}
Suppose that $\Omega \subset \Cb^d$ is a bounded strongly pseudoconvex domain with $\Cc^2$ boundary. Then the following are equivalent: 
\begin{enumerate}
\item the Kobayashi metric on $\Omega$ is a K\"ahler metric,
\item  the Kobayashi metric on $\Omega$ is a K\"ahler metric with constant holomorphic sectional curvature, 
\item the universal cover of $\Omega$ is biholomorphic to the unit ball.
\end{enumerate} 
\end{theorem}

As corollaries we obtain the following. 

\begin{corollary} Suppose that $\Omega \subset \Cb^d$ is a bounded strongly pseudoconvex domain with $\Cc^2$ boundary. Then the Bergman metric is a scalar multiple of the Kobayashi metric if and only if $\Omega$ is biholomorphic to the unit ball. 
\end{corollary} 

\begin{proof} If the Bergman metric is a scalar multiple of the Kobayashi metric, then Theorem~\ref{thm:spcv} implies that the Bergman metric has constant holomorphic sectional curvature. Then by a result of Lu~\cite{L1966}, $\Omega$ is biholomorphic to the unit ball. 

Conversely, if $\Omega$ is biholomorphic to the ball, then it is well known that the Bergman metric is a scalar multiple of the Kobayashi metric.
\end{proof} 

\begin{corollary} Suppose that $\Omega \subset \Cb^d$ is a bounded strongly pseudoconvex domain with $\Cc^2$ boundary. Then the K\"ahler-Einstein metric is a scalar multiple of the Kobayashi metric if and only if the universal cover of $\Omega$ is biholomorphic to the unit ball. 
\end{corollary} 

\begin{proof} If the K\"ahler-Einstein metric is a scalar multiple of the Kobayashi metric, then Theorem~\ref{thm:spcv} implies that the universal cover of $\Omega$ is biholomorphic to the unit ball. 

Conversely, if the universal cover $\wt{\Omega}$ of $\Omega$ is biholomorphic to the unit ball, then on $\wt{\Omega}$ the K\"ahler-Einstein metric is a scalar multiple of the Kobayashi metric. Since holomorphic covering maps between bounded domains are local isometries for both the  K\"ahler-Einstein metric and the Kobayashi metric, we see that on $\Omega$ the K\"ahler-Einstein metric is a scalar multiple of the Kobayashi metric.
\end{proof} 

\subsection{Ideas in the proofs}  One of the key ideas in the proof of Theorem~\ref{thm:metric_hartog's}  is to use a result of Kerner (see Theorem~\ref{thm:ker} below). This approach is motivated by earlier work of Nemirovski\u{\i}-Shafikov~\cite{NS2005,NS2005b} on uniformizing strongly pseudconvex domains with spherical boundary. 

The non-trivial part of Theorem~\ref{thm:spcv} consists in proving that (1) $\Rightarrow$ (3). For this direction, we first show that at a point sufficiently close to the boundary of a strongly pseudoconvex domain, there is an open set of directions where the Kobayashi and Carath\'eodory metrics agree. Next we observe that if the Kobayashi metric is K\"ahler, then the Kobayashi metric must have constant negative holomorphic sectional curvature in these directions. This is used to show that the Kobayashi metric has constant negative holomorphic sectional curvature outside a compact set and hence by Corollary~\ref{cor:universal_cover_is_a_ball} the universal cover is biholomorphic to the ball. We note that the first step in this argument is motivated by work of Huang~\cite{H1994b} while the second is motivated by work of Wong~\cite{W1977b} (see~\cite{S1983} for some corrections).

\subsection*{Acknowledgements} We are indebted to Stefan Nemirovski\u{\i} for comments on an earlier version of this paper. In particular, he suggested that Kerner's theorem \cite{K1961} could be used to improve our results and this was indeed the case. 

\section{Preliminaries}

We fix some basic notations: 
\begin{itemize}
\item Throughout the paper, $d$ is an integer satisfying $d \geq 2$ and $(z_1,\dots,z_d)$ denotes the standard coordinates in $\Cb^d$. 
\item When the context is clear, we use $\norm{\cdot}$ to denote the Euclidean norm on $\Cb^d$. Then define $\Bb(z_0,r):=\{ z \in \Cb^d : \norm{z-z_0} < r\}$, $\Bb := \Bb(0,1)$, and  let $\Db \subset \Cb$ denote the unit disk. 
\item Given a domain $\Omega \subset \Cb^d$ we let $k_\Omega$ denote the Kobayashi infinitesimal pseudo-metric and let $K_\Omega$ denote the Kobayashi pseudo-distance obtained by integrating $k_\Omega$ along piecewise smooth curves. When $\Omega$ is bounded, $k_\Omega$ is non-degenerate and $K_\Omega$ is a distance. We normalize the Kobayashi metric so that $k_{\Bb}$ has constant holomorphic sectional curvature equal to -4.
\end{itemize}

\subsection{The topology of Stein manifolds} In this section we record some basic topological properties of Stein manifolds. 

We first recall that Stein manifolds are one ended, for a proof see~\cite[page 227]{GR2009}.

\begin{theorem}\label{thm:one-ended} If $X$ is a connected Stein manifold with $\dim_{\Cb} X \geq 2$, then $X$ is one ended (that is, if $K \subset X$ is a compact set, then there exists a compact set $K^\prime \subset X$ such that $K \subset K^\prime$ and $X \setminus K^\prime$ is connected). \end{theorem}

We also use the following  (probably well known) fact about the homotopy groups of a Stein manifold. 

\begin{proposition}\label{prop:fundamental groups} Suppose that $X$ is a connected Stein manifold with $\dim_{\Cb} X \geq 2$ and $K \subset X$ is compact. If $1 \leq k < \dim_{\Cb} X$, then the inclusion map $X \setminus K \hookrightarrow X$ induces a surjective map $\pi_k(X \setminus K) \rightarrow \pi_k(X)$. 
\end{proposition}

The following argument is similar to the proof of ~\cite[Lemma 5.2]{HJ1998}. 

\begin{proof} We can assume that $X \subset \Cb^m$ is a closed complex submanifold. Then there exists $z_0 \in \Cb^m$ such that the function $f : X \rightarrow \Rb$ given by 
$$
f(z) = \norm{z-z_0}^2
$$
is Morse (see~\cite[Theorem 6.6]{M1963}). Further, since $f$ is strictly plurisubharmonic, the index of each critical point is at most $\dim_{\Cb} X$ (see~\cite[Section 7]{M1963}). Let $\Psi_t : X \rightarrow X$ denote the flow associated to the gradient of $f$. Since 
$$
\norm{\nabla f(z)} \leq 2\norm{z-z_0},
$$
Gr\"onwall's inequality implies that the flow exists for all time. 

If $p \in X$ is a critical point of $f$, then 
$$
W^+(p) := \left\{ z \in X : \lim_{t \rightarrow \infty} \Psi_t(z) = p\right\}
$$
is a submanifold diffeomorphic to $\Rb^{\lambda(p)}$ where $\lambda(p)$ is the index of $p$ (see the proof of~\cite[Proposition 2.24]{N2007}). 

Fix $1 \leq k < \dim_{\Cb} X$ and suppose that $\sigma_0 : \mathbb{S}^k \rightarrow X$ is continuous. We claim that $\sigma_0$ is homotopic to a map $\sigma_1 : \mathbb{S}^k \rightarrow X\setminus K$. Let $p_1,\dots, p_k$ denote the critical points of $f$ in $K$. By perturbing $\sigma_0$ we can assume that $\sigma_0$ is transverse to each $W^+(p_1), \dots, W^+(p_k)$. Then, by dimension counting $\sigma_0(\mathbb{S}^k) \cap W^+(p_j) = \emptyset$ for all $1 \leq j \leq k$. Then for $T$ sufficiently large $\sigma_1 = \Psi_T \circ \sigma_0$ has image in $X \setminus K$ and $H(t,\cdot) = \Psi_{tT} \circ \sigma_0$ is a homotopy from $\sigma_0$ to $\sigma_1$. 

So the inclusion map $X \setminus K \hookrightarrow X$ induces a surjective map $\pi_k(X \setminus K) \rightarrow \pi_k(X)$.
\end{proof} 

\subsection{Envelope of holomorphy and Kerner's theorem} 

In this subsection we recall Kerner's theorem. We follow the presentation in \cite{NS2005} and for more details on envelopes of holomorphy we refer the reader to~\cite[Chapter 1, Section G]{GR2009}. In this subsection, we also assume that all manifolds are connected. 

Let $D$ be a Riemann domain over a Stein manifold $X$, that is $D$ is a complex manifold and there exists a locally biholomorphic map $p_D: D \rightarrow X$ (which is not necessarily onto). There is a maximal domain $E(D)$ over $X$, called the {\sl envelope of holomorphy} of $D$, such that every holomorphic function on $D$ extends to a holomorphic function on $E(D)$. More precisely, the \emph{envelope of holomorphy} of $D$ is a pair $(E(D), \alpha_D)$ where $p_{E(D)}: E(D) \rightarrow X$ is a Riemann domain over $X$ and $\alpha_D : D \rightarrow E(D)$ is a locally biholomorphic map such that
\begin{enumerate}
\item $p_D = p_{E(D)} \circ \alpha_D$, 
\item for every holomorphic function $f : D  \rightarrow \Cb$, there exists a holomorphic function $F: E(D) \rightarrow \Cb$ such that $F \circ \alpha_D = f$,
\item if $(G, \beta)$ is another pair satisfying these properties, then there exists a locally biholomorphic map $f : G \rightarrow E(D)$ such that $f \circ \beta = \alpha_D$ and $p_{E(D)} =p_G \circ f$.  
\end{enumerate}
We also note that since $X$ is Stein, $E(D)$ is also Stein~\cite{R1963}. 

We will use the following extension result (see Theorem 2.9 in \cite{NS2005}).

\begin{lemma}\label{lem:ext}
Suppose that $X$ and $Y$ are Stein manifolds. If $D$ is a Riemann domain over $X$ and $f : D \rightarrow Y$ is a locally biholomorphic map, then there exists a locally biholomorphic map  $F : E(D) \rightarrow Y$ such that $F \circ \alpha_D = f$.
\end{lemma}

Given a Riemann domain $D$ over a Stein manifold $X$, let $\pi_D : \widetilde{D} \rightarrow  D$ denote the universal cover of $D$. Notice that, by definition, $\widetilde{D}$ is also a Riemann domain over $X$ and $\pi_D$ is a locally biholomorphic map. Kerner established the following remarkable connection between universal covers and envelopes of holomorphy.

\begin{theorem}\cite{K1961}\label{thm:ker}
If $D$ is a Riemann domain over a Stein manifold $X$, then the universal cover of $E(D)$ is biholomorphic to $E(\widetilde{D})$. Moreover, under this identification we have the following commutative diagram

\begin{center}
\begin{tikzcd}
\widetilde{D} \arrow[d, "\pi_D"'] \arrow[rr, "\alpha_{\widetilde{D}}"] & &  E(\widetilde{D})  \arrow[d, "\pi_{E(D)}"]\\
D \arrow[rr, "\alpha_D"] & & E(D)
\end{tikzcd}
\end{center}

\end{theorem}

\subsection{Symmetric spaces}\label{subsec:symmetric spaces} In this section we recall some properties of symmetric spaces, for more details see~\cite{H2001}.

Given Riemannian manifolds $(M,g)$ and $(N,h)$, a map $f : M \rightarrow N$ is called a \emph{local isometry} if it is a local diffeomorphism and 
$$
g=f^*h
$$
on $M$. If, in addition, $f$ is a diffeomorphism, then $f$ is called an \emph{isometry}. 

A Riemannian manifold $(M,g)$ is called a \emph{locally symmetric space} if for every $p \in M$ there is an open neighborhood $U_p$ of $p$ and a local isometry $s_p : U_p \rightarrow M$ such that $s_p(p) = p$ and $d(s_p)_p = -\id$. If each $s_p$ is an isometry defined on all of $M$ and $M$ is connected, then $M$ is a \emph{(global) symmetric space}. In both cases, the map $s_p$ is called a \emph{geodesic symmetry at $p$}. Any isometry of a connected Riemannian manifold is determined by its value and derivative at a point, so if $s_p : U_p \rightarrow M$ and $\hat{s}_p : \hat{U}_p \rightarrow M$ are geodesic symmetries at a point $p$, then $s_p = \hat{s}_p$ on the connected component of $U_p \cap \hat{U}_p$ containing $p$. So geodesic symmetries in a  locally symmetric space are locally unique and geodesic symmetries in a symmetric space are unique.

Since the isometry group of a symmetric space acts transitively, every symmetric space is complete. Further, every locally symmetric space has a real analytic structure (see~\cite[Chapter iV, Proposition 5.5]{H2001}) and so we will always assume that they are real analytic Riemannian manifolds. 

Given a Riemannian manifold $(M,g)$, let $B(p,r) \subset M$ denote the metric ball of radius $r > 0$ centered at $p \in M$. Recall that $B(p,r)$ is called \emph{normal} if the exponential map $\exp_p$ is defined on $\{ v \in T_p M : \norm{v} < r\}$ and induces a diffeomorphism between this set and $B(p,r)$. We will use the following observation. 

\begin{observation}\label{obs:symmetries on normal balls} If $(M,g)$ is a locally symmetric space and $B(p,r)$ is a normal ball, then we can assume that $s_p$ is defined on $B(p,r)$ and induces an isometry of $B(p,r) \rightarrow B(p,r)$. \end{observation} 

\begin{proof} Let
$$
f := \exp_p \circ (-\id_{T_pM}) \circ \exp_p^{-1} : B(p,r) \rightarrow B(p,r).
$$
Then $f$ is a real analytic diffeomorphism, since $(M,g)$ is a real analytic Riemannian manifold, with $f(p)=p$. Further, since $s_p$ maps geodesics to geodesics, $s_p = f$ in a neighborhood of $p$. In particular, $f^* g = g$ in a neighborhood of $p$. Since $f$ is real analytic, then $f^* g = g$ on all of $B(p,r)$ and so $f$ is an isometry. Since $d(\exp_p)_0 = \id$ under the identification $T_0 T_p M \simeq T_p M$, we have $d(f)_p = -\id_{T_pM}$. So we can assume that $s_p = f$. 
\end{proof} 

A Riemannian manifold $(M,g)$ is a \emph{Hermitian (locally) symmetric space} if $(M,g)$ is a (locally) symmetric space, $g$ is a Hermitian metric, and each geodesic symmetry is holomorphic. The proof of ~\cite[Chapter VIII, Proposition 4.1]{H2001} shows that in this case $g$ is actually a K\"ahler metric.

\section{Proof of Theorem~\ref{thm:metric_hartog's}}

Suppose that $X$ is a Stein manifold with $\dim_{\Cb} X \geq 2$, $K \subset X$ is a compact subset where $X \setminus K$ is connected, and $g_0$ is a locally symmetric Hermitian metric on $X \setminus K$ which is complete at infinity.  

Let $U: = X \setminus K$. 

\begin{lemma}\label{lem:local transitivity} If $z_1, z_2 \in U$, then there is an open neighborhood $\Oc$ of $z_1$ in $U$ and a holomorphic local isometry $f : \Oc \rightarrow U$ with $f(z_1) = z_2$. \end{lemma} 

\begin{proof} Fix a smooth curve $\sigma : [0,1] \rightarrow U$ with $\sigma(0)=z_1$ and $\sigma(1) = z_2$. Then fix $C > 1$ such that 
$$
\dist(\sigma(s), \sigma(t)) \leq C\abs{t-s}
$$
for all $s,t \in [0,1]$. By compactness, there exists $r_0 > 0$ such that $B(z,r_0)$ is a normal ball for all $z \in \sigma([0,1])$. Then there exists $0 < r < r_0$ such that $B(z,r)$ is a normal ball for all 
$$
z \in \bigcup_{w \in \sigma([0,1])} \overline{B(w,r_0/2)}. 
$$
Fix a partition 
$$
0 =  t_1 < \dots < t_{m+1} = 1
$$
such that $\abs{t_{j+1}-t_j} < \frac{2r}{C}$ for $1 \leq j \leq m$. Then for $1 \leq j \leq m$, let $p_j$ denote the midpoint of the unique geodesic joining $\sigma(t_j)$ to $\sigma(t_{j+1})$. By construction, $B(p_j,r)$ is a normal ball and 
$$
\sigma(t_j), \sigma(t_{j+1}) \in B(p_j,r). 
$$
Let $s_j : B(p_j,r) \rightarrow M$ be the geodesic symmetry at $p_j$ (see Observation~\ref{obs:symmetries on normal balls}). Then $s_j(\sigma(t_j)) = \sigma(t_{j+1})$. Finally the holomorphic local isometry
\begin{align*}
f  := s_m \circ s_{m-1} \circ \cdots \circ s_1 : \Oc \rightarrow M
\end{align*}
defined on 
$$
\Oc  := \bigcap_{k=1}^m (s_k \circ s_{k-1} \circ \cdots \circ s_1)^{-1}( B(p_{k+1},r))
$$
maps $z_1$ to $z_2$. 
\end{proof} 

\begin{lemma}\label{lem:local_model} There exists a simply connected Hermitian symmetric space $(M,h)$ such that: if $z \in U$, then there exist a neighborhood $\Oc_z$ of $z$ in $U$ and a holomorphic local isometry $\phi_z : \Oc_z \rightarrow M$. 
\end{lemma} 

This lemma is probably well known, but lacking a reference we provide a proof. 

\begin{proof} Using Lemma~\ref{lem:local transitivity}, it suffices to prove the lemma for a single $z_0 \in U$. By~\cite[Chapter 4, Theorem 5.1, Corollary 5.7]{H2001}, there exist a simply connected symmetric space $(M,h)$, a neighborhood $\Oc$ of $z_0$ in $U$, and a  local isometry $\phi : \Oc \rightarrow M$. We will construct a complex structure on $M$ that makes $(M,h)$ a Hermitian symmetric space and $\phi$ a holomorphic map.

Using~\cite[Chapter V, Proposition 4.2]{H2001} we can decompose 
$$
M = M_c \times M_{nc} \times M_{euc}
$$
where $M_c$ is a symmetric space of compact type, $M_{nc}$ is a symmetric space of non-compact type, and $M_{euc}$ is isometric to a Euclidean space. It is possible for some of these factors to be zero dimensional. 

Given $p \in M$, let $C_p \subset M$ denote the cut locus of $p$ in $M$. We claim that if $p_1, \dots, p_m \in M$, then the set
$$
\bigcap_{j=1}^m (M\setminus C_{p_j}) = M \setminus \bigcup_{j=1}^m C_{p_j}
$$
is connected. If $\dim_{\Rb} M_c = 0$, then $M$ has non-positive sectional curvature. So, in this case, $C_{p} = \emptyset$ for all $p \in M$ and the claim is obviously true. Suppose now that $\dim_{\Rb} M_c > 0$. In this case, if $p=(p_c, p_{nc},p_{euc}) \in M_c \times M_{nc} \times M_{euc}$, then 
$$
C_p = C_{p_{c}} \times M_{nc} \times M_{euc}
$$
where $C_{p_{c}}$ is the cut locus of $p_c$ in $M_c$.  Further, by  \cite[Corollary 3, Theorem 1.1]{T1978}, $C_{p_{c}}$ is a union of finitely many submanifolds of $M_{c}$, each with co-dimension at least two. Hence for any $p_1, \dots, p_m \in M$ the set 
$$
\bigcap_{j=1}^m (M\setminus C_{p_j}) = M \setminus \bigcup_{j=1}^m C_{p_j}
$$
is the complement of finitely many submanifolds, each with co-dimension at least two; hence it is connected. This proves the claim.

Let $J_0$ denote the complex structure on $X$. Recall, from Section~\ref{subsec:symmetric spaces}, that $g_0$ is K\"ahler and hence $J_0$ is invariant under parallel transport. 

By shrinking $\Oc$ we may assume that $\phi$ is a diffeomorphism onto its image. Then consider the complex structure $J = \phi_*J_0$ on $\phi(\Oc)$. Given $p \in \phi(\Oc)$, let $E_p = M \setminus C_p$. Since every point in $E_p$ is joined to $p$ by a unique minimal length geodesic, we can define an almost complex structure $J^p$ on $E_p$ via parallel transporting $J(p) 
\in \End(T_pM)$ along these minimal length geodesics. Since $J_0$ is invariant under parallel transport, we see that $J^p = J$ on $\phi(\Oc)$. Since the metric $h$ is real analytic, so is $J^p$.

Now, if $p,q \in \phi(\Oc)$, then by the remarks above $E_p \cap E_q = M \setminus (C_p \cup C_q)$ is connected and by definition contains $\phi(\Oc)$, so by real analyticity $J^p=J^q$ on $E_p \cap E_q$. Thus 
$$
\bigcup_{p \in \phi(\Oc)} J^p 
$$ 
defines an almost complex structure on $\cup_{p \in \phi(\Oc)} E_p$ which we also denote by $J$. Since $J=\phi_*J_0$ on $\phi(\mathcal O)$, the local isometry $\phi$ is holomorphic on $\mathcal O$.

Next we claim that $M = \cup_{p \in \phi(\Oc)} E_p$. Fix $p_0 \in \phi(\Oc)$ and $q \in M \setminus E_{p_0}$. Since $C_{p_0}$ is a union of submanifolds with positive codimension (again by \cite[Corollary 3, Theorem 1.1]{T1978}) and the map $g \in {\rm Isom}(M) \mapsto g(q) \in M$ is a submersion, there exists $g \in {\rm Isom}(M)$ arbitrarily close to the identity such that $g(q) \in E_{p_0}$. Then $q \in E_{g^{-1}(p_0)}$ and if $g$ is sufficiently close to the identity, then $g^{-1}(p_0) \in \phi(\Oc)$. So $M = \cup_{p \in \phi(\Oc)} E_p$ and hence $J$ is an almost complex structure on $M$.

To show that $J$ is a complex structure on $M$, it suffices by the Newlander-Nirenberg theorem to show for all vector fields $V,W$ on $M$ the associated Nijenhuis tensor 
 $$
N_J(V,W) := [V,W]+J[JV, W] + J[V, JW]-[JV, JW]
$$ 
vanishes identically. One can check that $N_J(V,W)$ is $\Cc^\infty(M)$-linear in $V$ and $W$. Hence if $q \in M$, then $N_J(V,W)(q)$ only depends on $V(q)$ and $W(q)$. Thus $N_J$ induces a function 
$$
N_J^q : T_q M \times T_q M \rightarrow T_q M
$$
defined by $N_J^q(v,w) = N_J(V,W)(q)$, where $V,W$ are any vector fields with $V(q) = v$ and $W(q) = w$. Since $J_0$ is a complex structure on $X$ and $J=\phi_*J_0$ on $\phi(\Oc)$, then $N_J^q \equiv 0$ for all $q \in \phi(\Oc)$. However, since $J$ is real analytic, $N_J^q$ is real analytic in $q$. Hence $N_J^q \equiv 0$ for all $q \in M$.

To show that $(M,h)$ is a Hermitian symmetric space, we need to show that each geodesic symmetry $s_p$ is holomorphic and 
\begin{equation}
\label{eqn:hJ = h}
h(J \cdot , J\cdot) = h.
\end{equation} 
Notice that both sides of Equation~\eqref{eqn:hJ = h} are real analytic and are equal on the open set $\phi(\Oc)$, so they are equal everywhere. Similarly,
$$
d(s_p)_q \circ J(q) - J(s_p(q)) \circ d(s_p)_q
$$
is real analytic in $p, q \in M$. Since it vanishes when $p, q \in \phi(\Oc)$, it must vanish everywhere. So each $s_p$ is holomorphic.  
 \end{proof}

\begin{lemma}\label{lem:local to global} Any holomorphic local isometry between connected open sets in $M$ extends to a global holomorphic isometry of $M$. \end{lemma}

\begin{proof} Let $f$ be a local holomorphic isometry between connected open sets in $M$. Since $M$ is a simply connected symmetric space, $f$  extends to  a real analytic isometry $F: M \rightarrow M$, see~\cite[Chapter 1, Proposition 11.4]{H2001}. To show that $F$ is holomorphic, notice that the equation 
$$
d(F)_p \circ J(p) - J(F(p)) \circ d(F)_p
$$
is real analytic with respect to $p \in M$ and vanishes when $p$ is contained in the domain of $f$, an open set. Thus it vanishes everywhere and $F$ is holomorphic. 
\end{proof} 

Let $\Aut(M,h)$ denote the holomorphic isometries of $(M,h)$. Then $\Aut(M,h)$ acts transitively on $M$ (see the discussion in~\cite[p. 372]{H2001} or combine the proof of Lemma~\ref{lem:local transitivity} with Lemma~\ref{lem:local to global}). 

For each $z \in U$, fix $\phi_z:\Oc_z \rightarrow M$ satisfying Lemma~\ref{lem:local_model}. By possibly replacing $\Oc_z$ with a smaller neighborhood we may assume that $\phi_z$ is a biholomorphism onto its image and $\Oc_z$ is convex (i.e. every two points in $\Oc_z$ are joined by a unique minimal length geodesic in $U$ and this geodesic is contained in $\Oc_z$). Notice that if $z_1,z_2 \in U$ and $\Oc_{z_1} \cap \Oc_{z_2} \neq \emptyset$, then $\Oc_{z_1} \cap \Oc_{z_2}$ is connected (by convexity) and the ``transition function'' 
$$
T_{z_2,z_1} := \phi_{z_2} \circ \phi_{z_1}^{-1}|_{\phi_{z_1}( \Oc_{z_1} \cap \Oc_{z_2})}
$$
is a holomorphic local isometry between connected open sets of $M$. Hence by Lemma~\ref{lem:local to global} there exists $\Phi \in \Aut(M,h)$ such that $T_{z_2,z_1}$ is the restriction of $\Phi$. In the language of geometric structures, see~\cite[Section 3.3]{T1997}, this means that the atlas $\{ (\Oc_z, \phi_z)\}_{z \in U}$ is a $(\Aut(M,h),M)$-structure on $U$.

 \begin{lemma} $M$ has no compact factors. \end{lemma} 
 
 \begin{proof} Using~\cite[Chapter VIII, Proposition 4.4]{H2001}, we can decompose 
 \begin{equation}
 \label{eqn:splitting part two}
M = M_c \times M_{nc} \times M_{euc}
\end{equation}
where $M_c$ is a Hermitian symmetric space of compact type, $M_{nc}$ is a Hermitian symmetric space of non-compact type, and $M_{euc}$ is isometric to complex Euclidean space. Suppose for a contradiction that $d:=\dim_{\Rb} M_c > 0$. Then $M_c$ is homogeneous and has positive Ricci curvature, so there exists $\epsilon > 0$ such that ${\rm Ric}(v) \geq \epsilon(d-1)$ for all $v \in T M_c$ with $\norm{v}=1$ (see~\cite[pp. 242 Remark 2]{H2001}).

The group ${\rm Isom}(M,h) \supset \Aut(M,h)$ preserves this splitting, see the proof of~\cite[Chapter VIII, Proposition 4.4]{H2001}, and so 
$$
\Theta_z := \phi_z^*T_z M_{c}
$$ 
defines an integrable subbundle of the tangent bundle $TU$ of $U$. Let $r = \frac{\pi}{\sqrt{\epsilon}}$ and fix $p \in U$ such that $\exp_p$ is defined on $\left\{ v \in T_p X : \norm{v} < r+2 \right\}$ (such a point exists since $g_0$ is complete at infinity). Let $\mathcal{M}$ denote the immersed connected submanifold through $p$ tangent to $\Theta$. Since $M_c \subset M$ is a complex submanifold and each $\phi_z$ is holomorphic, we see that $\mathcal{M}$ is a complex submanifold of $X$.

Endow $\mathcal{M}$ with the restriction of $g_0$ and let $\exp_p^{\mathcal{M}}$ be the associated exponential map at $p$. Since $\phi_z$ is a local isometry and the splitting in Equation~\eqref{eqn:splitting part two} is isometric, 
$$
\exp_p^{\mathcal{M}}(v) = \exp_p(v)
$$
for all $v \in T_p \mathcal{M}$ in the domain of $\exp_p$. In particular,  $\exp_p^{\mathcal{M}} $ is defined on $\left\{ v \in T_p \mathcal{M} : \norm{v} < r+2 \right\}$. Let $B_{\mathcal{M}}(p,r+1)$ denote the metric ball in $(\mathcal{M}, g)$ with radius $r+1$ centered at $p$. Then by the definition of the exponential map,
\begin{equation}
\label{eqn:image of exp map}
\overline{B_{\mathcal{M}}(p,r+1)} = \exp_p\left( \left\{ v \in T_p \mathcal{M} : \norm{v} \leq r+1 \right\}\right)
\end{equation}
is compact. 

Fix $q \in \Mcc$. We claim that $\dist_{\Mcc}(p,q) \leq r$. Suppose not. Since $\overline{B_{\mathcal{M}}(p,r+1)}$ is compact, we can find $q^\prime \in \overline{B_{\mathcal{M}}(p,r+1)}$ where $\dist_{\Mcc}(p,q^\prime) > r$ and 
$$
\dist_{\Mcc}(p,q^\prime) + \dist_{\Mcc}(q^\prime, q) = \dist_{\Mcc}(p,q)
$$
(notice that we do not know that $\Mcc$ is complete,  so we can't assume there is a minimum length geodesic joining $p$ to $q$). Then by Equation~\eqref{eqn:image of exp map} there exists a minimal length geodesic $\sigma$ joining $p$ to $q^\prime$. 

However, by construction $\Mcc$ is locally isometric to $M_c$ and so ${\rm Ric}_{\Mcc}(v) \geq \epsilon(d-1)$ for all $v \in T \Mcc$ with $\norm{v}=1$. Since $\sigma$ has length strictly greater than $r = \frac{\pi}{\sqrt{\epsilon}}$, the standard proof of the  Bonnet-Myers theorem using the index form shows that $\sigma$ is not length minimizing, see ~\cite[Theorem 1.31]{CE2008}. So we have a contradiction and hence $\dist_{\Mcc}(p,q) \leq r$.

Since $q \in \Mcc$ was arbitrary, we get
$$
\Mcc \subset \overline{B_{\mathcal{M}}(p,r+1)}
$$ 
and hence $\Mcc$ is compact. Then $\Mcc$ is a complex, compact, positive dimensional submanifold of $X$. Since $X$ is a Stein manifold,  we have a contradiction. 
  \end{proof}  
  
  \begin{lemma} $M$ is Stein. \end{lemma} 
  
  \begin{proof} Since $M$ has no compact factors, by the Harish-Chandra embedding $M$ is biholomorphic to $D \times \Cb^{m_2}$ where $D \subset \Cb^{m_1}$ is a bounded pseudoconvex domain and $m_1,m_2 \geq 0$ (see~\cite[Chapter VIII, Section 7]{H2001}). 
  \end{proof}

Let $\pi_U: \widetilde{U} \rightarrow U$ denote the universal cover of $U$. Then let ${\rm dev} : \wt{U} \rightarrow M$ denote a developing map associated to the  $( \Aut(M,h),M)$-structure on $U$ (see~\cite[Section 3.4]{T1997}). By definition, ${\rm dev}$ is a local diffeomorphism with the property: if $w \in \wt{U}$, then there exists some $\Phi_w \in \Aut(M,h)$ with 
\begin{equation}
\label{eqn: developing map}
\phi_{\pi_U(w)} \circ \pi_{U} = \Phi_w \circ {\rm dev}
\end{equation}
in a neighborhood of $w$. Since $\phi_{\pi_U(w)}$, $\pi_U$, and $\Phi_w$ are all local biholomorphisms, so is ${\rm dev}$. Moreover, $\dev$ is unique up to post composition with an element of $\Aut(M,h)$ (again see~\cite[Section 3.4]{T1997}).

Let $\pi_X: \wt{X} \rightarrow X$ denote the universal cover of $X$. By Hartogs' theorem, the envelope of holomorphy of $U$ is $X$ and so, by Theorem~\ref{thm:ker}, we can identify $E(\widetilde{U})=\widetilde{X}$. Then by Lemma~\ref{lem:ext}, there exists a local biholomorphism $F: \wt{X} \rightarrow M$ such that
\begin{equation}
\label{eqn:defn of F}
F \circ \alpha_{\widetilde{U}} = \dev.
\end{equation}
This is summarized by the following diagram
\begin{center}
\begin{tikzcd}
&  M & \\
\widetilde{U} \arrow[d, "\pi_U"'] \arrow[ru,"\dev"]  \arrow[rr, "\alpha_{\widetilde{U}}"] & &  \widetilde{X} =E(\widetilde{U}) \arrow[lu, "F"'] \arrow[d, "\pi_X"]\\
U \arrow[rr, hook] & & X= E(U)
\end{tikzcd}
\end{center}

Next identify $\pi_1(X)$ (respectively $\pi_1(U)$) with the deck transformation group of the covering $\pi_X: \wt{X} \rightarrow X$ (respectively $\pi_U: \widetilde{U} \rightarrow U$).

\begin{lemma}\label{lem:equivariance} There exists a homomorphism $\rho : \pi_1(X) \rightarrow \Aut(M,h)$ such that 
$$
F \circ \gamma = \rho(\gamma) \circ F
$$
for all $\gamma \in \pi_1(X)$. 
\end{lemma} 

\begin{proof} Fix $\gamma \in \pi_1(X)$. By Proposition~\ref{prop:fundamental groups} there exists $\hat{\gamma} \in \pi_1(U)$ such that $\gamma \circ \alpha_{\widetilde{U}}= \alpha_{\widetilde{U}} \circ \hat{\gamma}$. Notice that $\dev \circ \hat{\gamma}$ is a developing map for the geometric structure on $U$. So by the uniqueness property of the developing map, there exists $\rho(\gamma) \in \Aut(M,h)$ such that $\rho(\gamma) \circ \dev = \dev \circ \hat{\gamma}$. Then using Equation~\eqref{eqn:defn of F} we have 
$$
F \circ \gamma \circ \alpha_{\widetilde{U}} = F \circ \alpha_{\widetilde{U}} \circ \hat{\gamma} = \dev \circ \hat{\gamma} = \rho(\gamma) \circ \dev = \rho(\gamma) \circ F \circ \alpha_{\widetilde{U}}.
$$
Since $\alpha_{\widetilde{U}}$ has open image in $\wt{X}$ and the maps are holomorphic, this implies that $F \circ \gamma = \rho(\gamma) \circ F$. 

Since $F$ has open image in $M$, the element $\rho(\gamma)$ is unique. Then since 
$$
\rho(\gamma_1\gamma_2) \circ F = F \circ \gamma_1 \circ \gamma_2 = \rho(\gamma_1)\rho(\gamma_2) \circ F
$$
we see that $\rho$ is indeed a homomorphism. 
\end{proof} 

Now consider the K\"ahler metric $F^* h$ on $\wt{X}$. By Lemma~\ref{lem:equivariance} this descends to a metric $g$ on $X$. By construction $g$ is locally symmetric and $g=g_0$ on $U = X \setminus K$. Since $g_0$ is complete at infinity,  $g$ is complete.

\section{Proof of Corollary~\ref{cor:universal_cover_is_a_ball}}

Suppose that $X$ is a Stein manifold with $\dim_{\Cb} X \geq 2$, $K \subset X$ is compact, and $g_0$ is a complete K\"ahler metric on $X$ with constant negative holomorphic sectional curvature on $X \setminus K$. Since $X$ is one ended, see Theorem~\ref{thm:one-ended}, we may assume that $X \setminus K$ is connected. 

Using the Cartan-Ambrose-Hicks theorem we observe the following:

\begin{lemma} $g_0$ is locally symmetric on $X \setminus K$. 
\end{lemma} 

Since the argument is well known we only sketch the proof. 

\begin{proof}[Proof sketch]  Fix $z \in X \setminus K$ and let $\exp_{z} : T_z X \rightarrow X$ denote the exponential map at $z$ associated to the K\"ahler metric $g_0$. Then fix a normal ball $B(z,r)$ centered at $z$ with $B(z,r) \subset X \setminus K$ and define
$$
s_z : = \exp_z \circ (-\id_{T_z X}) \circ \exp_z^{-1} : B(z,r) \rightarrow B(z,r).
$$ 
The map $s_z$ is a diffeomorphism with $s_z(z) = z$. Since the holomorphic sectional curvature of $g_0$ is constant on $X \setminus K$, the Cartan-Ambrose-Hicks theorem, see Theorem 1.42 in~\cite{CE2008}, implies that $s_z$ is a holomorphic local isometry (we note that this argument requires that parallel transport commutes with complex multiplication and hence it is essential that $g_0$ is a K\"ahler metric instead of just Hermitian). Since $d(\exp_z)_0 = \id$ (under the identification $T_0 T_z X \simeq T_z X$), we have $d(s_z)_z = -\id_{T_z X}$. So $g_0$ is locally symmetric at $z$. 
\end{proof} 

Now, by Theorem~\ref{thm:metric_hartog's}, there exists a complete locally symmetric metric $g$ on $X$ such that $g = g_0$ on $X \setminus K$. Since the universal cover of $(X,g)$ is globally symmetric (hence homogeneous) and $g$ has constant negative holomorphic sectional curvature on $X \setminus K$, we then see that $g$ has constant negative holomorphic sectional curvature on $X$. So a classical result of Hawley~\cite{H1953} and Igusa~\cite{I1954} implies that the universal cover of $X$ is biholomorphic to the unit ball $\Bb$ in $\Cb^{\dim_{\Cb} X}$.

\section{Complex geodesics in strongly pseudoconvex domains} 

In this section we observe that in a strongly pseudoconvex domain and at points sufficiently close to the boundary, the directions where the Kobayashi and Carath\'eodory metrics agree has non-empty interior.

First we fix some notation and terminology. Given a bounded domain $\Omega \subset \Cb^d$ and  $z \in \Omega$, we let
$$
\delta_\Omega(z) := \min\{ \norm{z-x} : x \in \partial \Omega\}
$$ 
denote the distance from $z$ to $\partial \Omega$. If, in addition, $\Omega$ has $\Cc^1$  boundary and $x \in \partial \Omega$, we let  $T_{x} \partial \Omega \subset \Cb^d$ denote the real tangent space to $\partial \Omega$ at $x$ and let $T_{x}^{\Cb} \partial \Omega := T_{x} \partial \Omega \cap i T_{x} \partial \Omega$. Then, when  $\Omega \subset \Cb^d$ has $\Cc^2$  boundary  and $z \in \overline{\Omega}$ is sufficiently close to $\partial \Omega$, there is a unique point $\pi(z) \in \partial \Omega$ such that $\|z-\pi(z)\| = \delta_\Omega(z)$. In this case, we let 
$$
P_z : \Cb^d \rightarrow T_{\pi(z)}^{\Cb} \partial \Omega
$$ 
denote the (Euclidean) orthogonal projection and define $P_z^{\bot} := \Id - P_z$.

Given a domain $\Omega \subset \Cb^d$, a holomorphic map $\varphi : \Db \rightarrow \Omega$ is a \emph{complex geodesic}  if 
$$
K_\Omega(\varphi(z_1), \varphi(z_2)) = K_{\Db}(z_1, z_2)
$$
for all $z_1, z_2 \in \Db$. Then given  $z \in \Omega$ let 
$$E_\Omega(z) \subset  \Cb^d \simeq T_z \Omega$$
denote the set of non-zero vectors $v \in \Cb^d$ where there exist a complex geodesic $\varphi : \Db \rightarrow \Omega$ and a holomorphic map $\rho : \Omega \rightarrow \Db$ with $\varphi(0)=z$, $\varphi^\prime(0)\lambda = v$ for some $\lambda \in \Cb$, and $\rho \circ \varphi = \id_{\Db}$. Equivalently, $E_\Omega(z)$ is the set of non-zero tangent vectors at $z$ where the Kobayashi and Carath\'eodory metrics agree.

\begin{theorem}\label{thm:tangential geodesics} Suppose that $\Omega \subset \Cb^d$ is a bounded strongly pseudoconvex domain with $\Cc^2$ boundary.  There is some $\varepsilon > 0$ such that: if $z \in \Omega$ and $\delta_\Omega(z) < \varepsilon$, then $E_\Omega(z)$ has non-empty interior. \end{theorem} 

\begin{remark} In the case when $\Omega$ has $\Cc^3$ boundary, uniform estimates on the size of $E_\Omega(z)$ are known. In particular, there exists some $\varepsilon > 0$ such that: if $z \in \Omega$ and $\delta_\Omega(z) < \varepsilon$, then
$$
\left\{ v \in T_z \Omega : \norm{P_z^{\bot}(v)} < \varepsilon \norm{P_z(v)}\right\} \subset E_\Omega(z). 
$$
In this case the existence of $\varphi$ is Theorem 1 in~\cite{H1994b} and the existence of $\rho$ Theorem 1.1 in~\cite{BFW2019}.  Also, see~\cite[Proposition 4.3]{BK1994} for a similar result when $\Omega$ has $\mathcal C^6$ boundary.
\end{remark}

To prove Theorem~\ref{thm:tangential geodesics}, we will use the following embedding theorem which is etablished in~\cite[Theorem 2.6]{BFW2019}, using a result from~\cite{DFW2014} and techniques from~\cite{F1976} (similar embedding theorems can also be found in~\cite[Lemma 5]{H1994b} and~\cite[Lemma 4.4]{BK1994}).

\begin{theorem}\label{thm:embedding theorem} Suppose that $\Omega \subset \Cb^d$ is a bounded strongly pseudoconvex domain with $\Cc^k$ boundary, $k \geq 2$. For any $p \in \partial \Omega$ there exists a bounded strongly convex domain $D \subset \Cb^d$ with $\Cc^k$ boundary and a holomorphic embedding $\Phi : \Omega \rightarrow \Cb^d$ such that 
\begin{enumerate}
\item $\Phi$ extends to a $\Cc^k$ embedding on $\overline{\Omega}$, 
\item $\Phi(\Omega) \subset D$, and 
\item there exists an open neighborhood $V$ of $\Phi(p)$ such that $V \cap \Phi(\Omega) = V \cap D$. 
\end{enumerate}
\end{theorem} 

We also recall the following well known properties of complex geodesics in strongly convex domains. 

\begin{theorem}\label{thm:cplx geod basic properties} Suppose that $D \subset \Cb^d$ is a bounded strongly convex domain with $\Cc^2$ boundary.  
\begin{enumerate}
\item If $\varphi: \Db \rightarrow D$  is a complex geodesic, then there exists a holomorphic map $\rho : D \rightarrow \Db$ with $\rho \circ \varphi = \id_{\Db}$. 
\item For any $z \in D$ and non-zero $v \in \Cb^d$, there exists a unique complex geodesic $\varphi_{z,v} : \Db \rightarrow D$ with $\varphi_{z,v}(0) = z$ and $\varphi_{z,v}^\prime(0)\lambda = v$ for some $\lambda > 0$. 
\item Each complex geodesic $\varphi: \Db \rightarrow D$ extends to a continuous map $\overline{\varphi} : \overline{\Db} \rightarrow \overline{D}$. 
\item If $z_n \rightarrow z$ in $D$ and $v_n \rightarrow v$ in $\Cb^d\setminus\{0\}$, then $\overline{\varphi}_{z_n, v_n}$ converges uniformly to $\overline{\varphi}_{z,v}$ on $\overline{\Db}$.  
\end{enumerate} 
\end{theorem} 

\begin{proof} Part (1) is a deep result of Lempert~\cite{L1981}, for an exposition see~\cite[Section 2.6]{A1989}. 

For a proof of part (2), see~\cite[Corollary 2.6.30]{A1989}. 

Part (3) can be deduced in several ways. For instance, using properties of the Kobayashi metric it is possible to use a lemma of Hardy-Littlewood to prove that any complex geodesic satisfies an estimate of the form
\begin{equation}
\label{eqn:HL_lemma}
\norm{\varphi(z)-\varphi(w)} \leq C \norm{z-w}^{1/2}
\end{equation}
where $C$ is a constant which only depends on the domain $D$ and $\delta_D(\varphi(0))$. This immediately implies that a complex geodesic extends continuously (for a nice exposition of this approach see~\cite{M1993}, also see the last remark in ~\cite{CCS1999} where the exponent $1/2$ is replaced by $\alpha$ for any $0 < \alpha < 1$). It is also possible to use the fact that the Kobayashi distance is Gromov hyperbolic and the identity map $D \rightarrow D$ extends to a homeomorphism of the Gromov boundary and Euclidean boundary~\cite{BB2000}. Then part (3) immediately follows from the definition of the Gromov boundary. 

Finally, in the context of part (4), the uniqueness in part (2) implies that $\varphi_{z_n,v_n}$ converges locally uniformly to $\varphi_{z, v}$ on $\Db$. Then the uniform estimate in Equation~\eqref{eqn:HL_lemma} implies that $\overline{\varphi}_{z_n, v_n}$ converges uniformly to $\overline{\varphi}_{z,v}$ on $\overline{\Db}$.  

\end{proof}

Using a rescaling argument, we will prove the following fact about complex geodesics in strongly convex domains.

\begin{lemma}\label{lem:tangential geodesics} Suppose that $D \subset \Cb^d$ is a bounded strongly convex domain with $\Cc^2$ boundary.  For any $q \in \partial D$ and any neighborhood $V$ of $q$ in $\Cb^d$, there is some $\delta > 0$ with the following property: if $z \in D$, $\norm{z-q} < \delta$, and $\varphi : \Db \rightarrow D$ is a complex geodesic with $\varphi(0) = z$ and $P_z^{\bot}(\varphi^\prime(0)) = 0$, then $\varphi(\Db) \subset V$. 
\end{lemma}

Delaying the proof of Lemma~\ref{lem:tangential geodesics} for a moment, we prove Theorem~\ref{thm:tangential geodesics}.

\begin{proof}[Proof of Theorem~\ref{thm:tangential geodesics}] It suffices to fix $p \in \partial \Omega$ and find a neighborhood $\Oc_p$ of $p$ in $\Cb^d$ such that: if $z \in \Omega \cap \Oc_p$, then $E_\Omega(z)$ has non-empty interior. 

Fix $p \in \partial \Omega$ and let $D$, $\Phi$, and $V$ be given by Theorem~\ref{thm:embedding theorem}. Then let $\delta > 0$ be the constant from Lemma~\ref{lem:tangential geodesics} associated to $V$ and $\Phi(p) \in \partial D$. Let $\Oc_p$ be a neighborhood of $p$ in $\Cb^d$ such that 
$$
\norm{\Phi(z) - \Phi(p)} < \delta
$$
for all $z \in \Oc_p \cap \overline{\Omega}$. 

Now fix $z \in \Oc_p \cap \Omega$ and let $\hat{z} := \Phi(z)$. Let $\hat{E} \subset \Cb^d$ denote the set of non-zero vectors $v \in \Cb^d$ where the complex geodesic $\varphi_{\hat{z},v}: \Db \rightarrow D$ given in Theorem~\ref{thm:cplx geod basic properties} part (3) satisfies 
$$
\varphi_{\hat{z},v}(\Db) \subset V.
$$
Notice that Lemma~\ref{lem:tangential geodesics} implies that $\hat{E}$ is non-empty and Theorem~\ref{thm:cplx geod basic properties} part (4) implies that $\hat{E}$ is open. So it suffices to show that $d(\Phi)_z^{-1}\hat{E} \subset E_\Omega(z)$

Fix $v \in \hat{E}$. Then $\varphi_{\hat{z},v}(\Db) \subset V \cap D \subset \Phi(\Omega)$ and hence $\varphi:=\Phi^{-1} \circ \varphi_{\hat{z},v}$ is a well defined holomorphic map into $\Omega$. Further, if $\hat{\rho} : D \rightarrow \Db$ is holomorphic and $\hat{\rho} \circ \varphi_{\hat{z},v} = \id_{\Db}$, then $\rho = \hat{\rho} \circ \Phi : \Omega \rightarrow \Db$ is holomorphic and $\rho \circ \varphi = \id_{\Db}$. Then $\varphi$ is a complex geodesic. Finally, since 
$$
d(\Phi)_z^{-1}v = d(\Phi)_z^{-1} \varphi_{\hat{z},v}^\prime(0) \lambda
$$
for some $\lambda > 0$, we have $d(\Phi)_z^{-1}v \in E_\Omega(z)$. So $d(\Phi)_z^{-1}\hat{E} \subset E_\Omega(z)$.

\end{proof}

\subsection{Proof of Lemma~\ref{lem:tangential geodesics}} The proof of Lemma~\ref{lem:tangential geodesics} is a simple rescaling argument. 

Fix $q \in \partial D$ and a neighborhood $V$ of $q$. Suppose, to obtain a contradiction, that for every $n \in \Nb$ there exist $w_n \in D$ and a complex geodesic $\varphi_n : \Db \rightarrow D$ with $\norm{w_n-q}< 1/n$, $\varphi_n(0)=w_n$, $P^{\bot}_{w_n}(\varphi_n^\prime(0)) = 0$, and $\varphi_n(\Db) \not\subset V$. 

Let $T_n : \Cb^d \rightarrow \Cb^d$ be the translation $T_n(z) = z-\pi(w_n)$. Then, since $D$ is strongly convex, there exist $C,\epsilon > 0$ and for every $n \in \Nb$ a unitary matrix $U_n$ such that 
$$
\Bb(0,\epsilon) \cap U_nT_n(D) = \left\{ (z_1,\dots,z_d) \in \Bb(0,\epsilon) : {\rm Im}(z_1) > {\rm Re}(z_1)^2 + \sum_{j=2}^d \abs{z_j}^2 + R_n(z) \right\} 
$$
where $R_n : \Bb(0,\epsilon) \rightarrow \Rb$ satisfies
\begin{equation}
\label{eqn:error term}
\lim_{z \rightarrow 0} \sup_{n \geq 1} \frac{\abs{R_n(z)}}{\norm{z}^2}=0. 
\end{equation}

Next let $\delta_n := \norm{w_n-\pi(w_n)}$ and let $\Lambda_n$ be the diagonal matrix
$$
\Lambda_n = \begin{pmatrix} \delta_n^{-1} & & & \\ & \delta_n^{-1/2} & & \\ & & \ddots & \\ & & & \delta_n^{-1/2} \end{pmatrix}.
$$
Notice that $\Lambda_nU_nT_n(w_n) = ( i,0,\dots,0)$ and 
$$
\Bb(0,\delta_n^{-1/2}\epsilon) \cap \Lambda_n U_nT_n(D)  = \left\{ (z_1,\dots,z_d) \in \Bb(0,\delta_n^{-1/2}\epsilon) : {\rm Im}(z_1) > \delta_n{\rm Re}(z_1)^2 + \sum_{j=2}^d \abs{z_j}^2 + \hat{R}_n(z) \right\} 
$$
where $\hat{R}_n(z) = \delta_n^{-1} R_n(\Lambda_n^{-1}(z))$. Notice that $\norm{\Lambda_n^{-1}(z)} \leq \delta_n^{1/2} \norm{z}$ for every $z \in \Cb^d$. Then Equation~\eqref{eqn:error term} implies that for every $z \in \Cb^d$
$$
\lim_{n \rightarrow \infty} \abs{\hat{R}_n(z)} = \lim_{n \rightarrow 0} \frac{\abs{R_n(\Lambda_n^{-1}(z))}}{\delta_n} \leq \lim_{n \rightarrow 0} \frac{\abs{R_n(\Lambda_n^{-1}(z))}}{\norm{\Lambda_n^{-1}(z)}^2}\norm{z}^2 = 0
$$
and the convergence is locally uniform in $\Cb^d$. So the sequence of domains $D_n:=\Lambda_n U_n T_n(D)$ converges in the local Hausdorff topology to 
$$
\Pc =\left\{ (z_1,\dots,z_d) \in \Cb^d: {\rm Im}(z_1) >  \sum_{j=2}^d \abs{z_j}^2  \right\}.
$$

Notice that $\phi_n := \Lambda_n U_nT_n\varphi_n : \Db \rightarrow D_n$ is a complex geodesic with $\phi_n(0)=(i,0,\dots,0)$. So by Theorem 4.1 and Proposition 4.2 in~\cite{Z2016}, we may pass to a subsequence such that $\phi_n$ converges locally uniformly to a complex geodesic $\phi : \Db \rightarrow \Pc$. By construction $\phi_n^\prime(0) \in \{0\} \times \Cb^{d-1}$ for all $n$ and so $\phi^\prime(0) \in\{0\} \times \Cb^{d-1}$. It is well known that complex geodesics in $\Pc$ parametrize complex affine lines intersected with $\Pc$ and so 
$$
\phi(\Db) \subset (\{i\} \times \Cb^{d-1}) \cap \Pc = \left\{ (i, z_2, \dots, z_d) : \sum_{j=2}^d \abs{z_j} < 1\right\}
$$
and in particular $\phi(\Db)$ is bounded. 

By assumption, there exists $\lambda_n \in \Db$ such that $\varphi_n(\lambda_n) \notin V$. Then 
$$
\lim_{n \rightarrow \infty} \phi_n(\lambda_n) = \lim_{n \rightarrow \infty} \Lambda_n U_nT_n(\varphi_n(\lambda_n))= \infty
$$
in the one point compactification $\Cb^d \cup \{\infty\}$.  Consider the real geodesic $\sigma_n : \Rb \rightarrow D_n$ defined by 
$$
\sigma_n(t) = \phi_n\left( \tanh(t) \frac{\lambda_n}{\abs{\lambda_n}} \right). 
$$
Passing to another subsequence we can suppose that $\frac{\lambda_n}{\abs{\lambda_n}}  \rightarrow e^{i\theta} \in \mathbb{S}^1$. Then $\sigma_n$ converges to a real geodesic $\sigma : \Rb \rightarrow \Pc$ given by $\sigma(t) = \phi(\tanh(t) e^{i\theta})$. Then by Proposition 7.9 and Example 9.4 in~\cite{Z2016} 
$$
\lim_{t \rightarrow \infty} \sigma(t) = \lim_{n \rightarrow \infty} \sigma_n(\tanh^{-1}(\abs{\lambda_n})) = \infty
$$
in $\Cb^d \cup \{\infty\}$. This is a contradiction since $\sigma(\Rb) \subset \phi(\Db)$ and $\phi(\Db)$ is bounded. 

\section{Proof of Theorem~\ref{thm:spcv}}

Suppose that $\Omega \subset \Cb^d$ is a bounded strongly pseudoconvex domain with $\Cc^2$ boundary. If $d =1$, then the three conditions in Theorem~\ref{thm:spcv} are always true by the uniformization theorem. So suppose that $d \geq 2$. 

It is well known  that the Kobayashi metric on the unit ball is K\"ahler. Further, holomorphic covering maps are local isometries for the Kobayashi metric. Hence we see that (3) $\Rightarrow$ (2). Also, (2) $\Rightarrow$ (1) by definition. 

The rest of the section is devoted to the proof of (1) $\Rightarrow$ (3). So suppose that the Kobayashi metric $k_\Omega$ is a K\"ahler metric. Motivated by the work of Wong~\cite{W1977b} (see~\cite{S1983} for some corrections), we will show that the metric has constant holomorphic sectional curvature near the boundary. Then Corollary~\ref{cor:universal_cover_is_a_ball} will imply that the universal cover of $\Omega$ is biholomorphic to the unit ball $\Bb \subset \Cb^d$. 

Given $z \in \Omega$, let $E_\Omega(z) \subset \Cb^d$ be as in Theorem~\ref{thm:tangential geodesics}. Also, given $v \in T_z \Omega \simeq \Cb^d$ non-zero, let $H(z;v)$ denote the holomorphic sectional curvature at $v$. 

\begin{lemma}  If $z \in \Omega$ and $v \in E_\Omega(z)$, then $H(z;v) = -4$. \end{lemma}

\begin{proof} By definition there exist a complex geodesic $\varphi : \Db \rightarrow \Omega$ and a holomorphic map $\rho : \Omega \rightarrow \Db$ where $\varphi(0)=z$, $\varphi^\prime(0)\lambda = v$ for some non-zero $\lambda \in \Cb$, and $\rho \circ \varphi = \id_{\Db}$. 

By the monotonicity property of the Kobayashi distance, 
\begin{align*}
K_\Omega(\varphi(\zeta),\varphi(\eta)) & \leq K_{\Db}(\zeta,\eta)=K_{\Db}((\rho\circ \varphi)(\zeta),(\rho \circ \varphi)(\eta))  \leq K_\Omega(\varphi(\zeta),\varphi(\eta)) 
\end{align*}
for every $\zeta, \eta \in \Db$. Thus  $\varphi : (\Db, K_{\Db}) \rightarrow (\Omega, K_\Omega)$ is an isometric embedding and hence
\begin{equation*}
H(z;v) = -4. \qedhere
\end{equation*}
\end{proof} 

\begin{lemma} There exists a compact set $K \subset \Omega$  such that the holomorphic sectional curvature of $k_\Omega$ equals $-4$ on $\Omega \setminus K$. \end{lemma} 

\begin{proof} Fix $\varepsilon > 0$ satisfying Theorem~\ref{thm:tangential geodesics}. Then if $z \in \Omega$ and $\delta_\Omega(z)< \varepsilon$, then $H(z;v)$ equals $-4$ on the open set $E_\Omega(z)$. Since the map $v \in T_z \Omega \mapsto H(z;v) \in \Rb$ is $\Rb$-rational ($z$ is fixed), we obtain that $H(z;v) = -4$ for all non-zero $v \in T_z \Omega$. 

So if 
$$
K := \{ z \in \Omega : {\rm dist}_{\rm Euc}(z,\partial \Omega) \geq \epsilon\},
$$
then the holomorphic sectional curvature of $k_\Omega$ equals $-4$ on $ \Omega \setminus K$. \end{proof} 

Finally Corollary~\ref{cor:universal_cover_is_a_ball} implies that the universal cover of $\Omega$ is biholomorphic to the unit ball $\Bb \subset \Cb^d$. This completes the proof of (1) $\Rightarrow$ (3).

\bibliographystyle{plain}
\bibliography{complex_kob}

\end{document}